\documentclass[a4paper,12pt]{amsart}

\usepackage[top=1.5in, bottom=1in, left=0.9in, right=0.9in]{geometry}

\usepackage[utf8]{inputenc}
\usepackage[all]{xy}
\usepackage{color}
\usepackage{hyperref}

\usepackage{amsmath, mathtools}
\usepackage{hyperref}
\hypersetup{colorlinks=true}
\usepackage{extpfeil}
\usepackage{amssymb}
\usepackage{amsfonts}
\usepackage{graphicx}
\usepackage{mathrsfs}
\usepackage{mathtools}
\usepackage{amscd}
\usepackage[all]{xy}
\usepackage{hyperref}
\usepackage{amsthm}
\usepackage{cleveref}
\usepackage{verbatim}
\usepackage{amscd}
\usepackage{mathtools}
\usepackage{comment}
\usepackage{diagbox}

\newtheorem{thm}{\bf Theorem}[section]

\newtheorem{prop}[thm]{\bf Proposition}
\newtheorem{lemma}[thm]{\bf Lemma}
\newtheorem{cor}[thm]{\bf Corollary}
\theoremstyle{definition}

\theoremstyle{remark}
\newtheorem{remark}[thm]{\bf Remark}
\newtheorem{Notation}[thm]{\bf Notation}

\newtheorem{question}[thm]{\bf Question}
\newtheorem{example}[thm]{\bf Example}
\numberwithin{equation}{section}

\newcommand{\HH}[3]{\operatorname{H}^{#1}_{#2}(#3)}

\DeclareMathOperator{\height}{{ht}}

\DeclareMathOperator{\depth}{{depth}}

\DeclareMathOperator{\Spec}{{Spec}}

\DeclareMathOperator{\Ass}{{Ass}}

\DeclareMathOperator{\Ext}{{Ext}}
\DeclareMathOperator{\Hom}{{Hom}}
\DeclareMathOperator{\Ker}{{{Ker}}}

\DeclareMathOperator{\chara}{{{char}}}

\DeclareMathOperator{\Sing}{{{Sing}}}
\DeclareMathOperator{\codim}{{{codim}}}

\DeclareMathOperator{\Proj}{{Proj}}
\DeclareMathOperator{\indeg}{{indeg}}
\DeclareMathOperator{\topdeg}{{topdeg}}
\DeclareMathOperator{\reg}{{reg}}

\DeclareMathOperator{\Sym}{{Sym}}

\def\ls{\leqslant}
\def\gs{\geqslant}

\def\f0{\mathbf{0}}

\def\fx{\mathbf{x}}

\def\fp{\mathbf{p}}

\def\fa{\mathbf{a}}

\def\fm{\mathfrak{m}}
\def\fn{\mathfrak{n}}

\def\fp{\mathfrak{p}}
\def\ffq{\mathfrak{q}}

\def \ZZ{\mathbb Z}
\def \NN{\mathbb N}

\def \C{\mathcal C}

\def \R{\mathcal R}

\def \M{\mathcal M}

\begin{document}

\title[On asymptotic vanishing behavior of local cohomology]{On asymptotic vanishing behavior of local cohomology}

\author[Hailong Dao]{Hailong Dao}
\address{Hailong Dao\\ Department of Mathematics \\ University of Kansas\\405 Snow Hall, 1460 Jayhawk Blvd.\\ Lawrence, KS 66045}
\email{hdao@ku.edu}

\author[Jonathan Monta\~no]{Jonathan Monta\~no}
\address{Jonathan Monta\~no \\ Department of Mathematical Sciences  \\ New Mexico State University  \\PO Box 30001\\Las Cruces, NM 88003-8001}
\email{jmon@nmsu.edu}

  \begin{abstract}
  Let $R$ be a standard graded algebra over a field $k$, with irrelevant maximal ideal $\fm$, and $I$ a homogeneous $R$-ideal. We study the asymptotic vanishing behavior of the graded components of the local cohomology modules $\{\HH{i}{\fm}{R/I^n}\}_{n\in \NN}$  for $i<\dim R/I$. We show that, when $\chara k= 0$, $R/I$ is Cohen-Macaulay, and $I$ is a complete intersection locally on $\Spec R \setminus\{\fm\}$, the lowest degrees of the modules $\{\HH{i}{\fm}{R/I^n}\}_{n\in \NN}$ are bounded by a linear function whose slope is controlled by the generating degrees of the dual of $I/I^2$. Our result is a direct consequence of a related bound for symmetric powers of locally free modules. If no assumptions are made on the ideal or the field $k$, we show that the complexity of the sequence of lowest degrees is at most polynomial, provided they are finite. Our methods also provide a result on stabilization of maps between local cohomology  of consecutive powers of ideals.
  \end{abstract}


\maketitle
%\tableofcontents

\section{Introduction}
Let $R$ be a Noetherian standard graded algebra over a field $k=R_0$, $\fm=\oplus_{n\in \ZZ_{> 0}}R_n$, and $I$ a homogeneous $R$-ideal. In this paper we study asymptotic behavior of the lowest degree of the local cohomology modules $\{\HH{i}{\fm}{R/I^n}\}_{n\in \NN}$, provided that they are finite.  As we make clear below, such behavior can be viewed as an ``asymptotic Kodaira vanishing for thickenings" phenomenon, and have recently appeared in various works such as \cite{BBLSZ, DM, Claudiu}.

To describe our motivations and questions precisely, let us recall some notations. For a graded $R$-module $M=\oplus_{i\in \ZZ}M_i$ one  defines $$\indeg{M}=\min\{i\mid M_i\neq0\},\qquad \topdeg{M}=\max\{i\mid M_i\neq 0\}.$$
If $M=0$, we set $\indeg M=\infty$ and $\topdeg M=-\infty$. We also set 
$$\beta(M)=\max\{i\mid (M/\fm M)_i\neq0\},$$
i.e., $\beta(M)$ is the maximal degree of an element in a minimal set of  homogeneous generators of $M$. 
The {\it Castelnuovo-Mumford regularity} of $M$ is defined as $$\reg(M)=\max\{\topdeg{\HH{i}{\fm}{M}}+i\}.$$
It is known that $\reg(R/I^n)$ agrees with a  linear function for $n\gg 0$, this fact was proved independently in \cite{CHT} and \cite{Kod} when $R$ is a polynomial ring over a field, and extended in \cite{TW} for arbitrary standard graded rings.

If $\HH{i}{\fm}{M}\neq 0$, the $\topdeg$ of $\HH{i}{\fm}{M}$ is always finite, however this is not the case for $\indeg$. In fact, since $\HH{i}{\fm}{M}$ is an Artinian module, we have that $\indeg\HH{i}{\fm}{M}>-\infty$ if and only if  $\HH{i}{\fm}{M}$ is Noetherian. Our work is guided by  the following questions raised in \cite{DM}.

\begin{question}\label{motivQ}
Assume $\HH{i}{\fm}{R/I^n}$ is Noetherian for $n\gg 0$.
\begin{enumerate}
\item Does there exist $\alpha \in \ZZ$ such that $\indeg \HH{i}{\fm}{R/I^n} >\alpha n$ for every $n\gg 0$? In other words, is $\displaystyle\liminf_{n\rightarrow \infty} \frac{ \indeg\HH{i}{\fm}{R/I^n}}{n}$ finite?
\item If so,  does the limit $\displaystyle\lim_{n\rightarrow \infty} \frac{ \indeg\HH{i}{\fm}{R/I^n}}{n}$ exist?
%\item Relate this type of vanishing with the core, as in Hyry-Smith.
\end{enumerate}

\end{question}

It follows from  \cite[1.4]{BBLSZ} that when $R$ is a polynomial ring over a field of characteristic $0$, $I$ is a prime ideal,  $X:=\Proj(R/I)$ is locally a complete intersection (lci), and $i$ is at most the codimension of the singular locus of $X$, then $ \indeg\HH{i}{\fm}{R/I^n} \gs 0$ for all $n>0$. As explained there, this can be viewed as a Kodaira Vanishing Theorem for thickenings of $X$. When $I$ is a determinantal ideal,  more precise behavior of vanishing results, and  other homological invariants of thickenings of $I$ are available, see for instance \cite{Claudiu} and \cite{Raicu2}.

Our initial interest in the question came from \cite{DM} where we need an affirmative answer to part (1) of Question \ref{motivQ} to obtain efficient bounds on the lengths of  local cohomology modules of powers. Robert Lazarsfeld pointed out to us that this is indeed the case when  $X=\Proj(R/I)$ is a l.c.i variety, $k$ is of characteristic zero,  and $i$ is at most the dimension of $X$. Thus Kodaira vanishing may not hold, but the lowest degrees of $\HH{i}{\fm}{R/I^n}$ are still bounded below by a linear function. 

In this work we provide further answers to Question \ref{motivQ} above, in the case when $R$ is not necessarily a polynomial ring and $I$ may not be prime, or even reduced. The first main general  result of this article (Theorem \ref{mainVB}), provides a linear lower bound for the initial degrees of local cohomology of the symmetric powers $\{S^n(E)\}_{n\in\NN}$, where  $E$ is a  graded module that is locally free on $\Spec R \setminus \{\fm\}$. Our proof relies on a duality statement and the result on regularity by Trung and Wang in \cite{TW}.  Here $(-)^*$ denotes the $R$-dual $\Hom_R(-,R)$.

\begin{thm} \label{MainT1}
Let $(R,\fm)$ be a standard Cohen-Macaulay graded algebra over a field $k$ of characteristic zero. Set $d=\dim R\gs  2$. Let $E$ be a  graded $R$-module which is free locally on $\Spec R\setminus \{\fm\}$.
Then there exists an integer $\varepsilon$ such that $$\indeg \HH{i}{\fm}{S^n(E)} \gs -\beta(E^*) n+\varepsilon$$ for every $n\gs 1$ and $ 1\ls i< d$.
\end{thm}

We apply this Theorem to answer Question \ref{motivQ} (1) affirmatively and effectively when $R$ is any standard graded algebra over a field of characteristic $0$, $R/I$ is Cohen-Macaulay, and $I$ is a complete intersection locally on $\Spec R\setminus \{\fm\}$. In this case,  $\liminf_{n\rightarrow \infty} \frac{ \indeg\HH{i}{\fm}{R/I^n}}{n}$ is bounded below by $-\max\{\beta(E^*),0\}$ where $E$ is the conormal module $I/I^2$ (see Corollary \ref{conormal}). This result can be seen as an  algebraic version of \cite[1.4]{BBLSZ} and  \cite[5.6]{DM},  our proof via local cohomology of symmetric powers of conormal modules is inspired by the proofs of these results.

Theorem \ref{MainT1} also allows us to show a result  on stabilization of maps of local cohomology of powers of ideals. We show that, if $I$ is as in the previous paragraph, and if $\beta(E^*)<0$, then the maps between local cohomology and Ext modules of consecutive powers of $I$ eventually stabilize on each graded degree.  This result is closely related to  \cite[1.1]{BBLSZ} and provide a partial answer to a question of Eisenbud, Musta\c{t}\u{a}, and Stillman  (\cite[6.1]{EMS}), see Corollary  \ref{stabExt} and its preceding paragraph for more details. 

In general, if one only assumes that $\HH{i}{\fm}{R/I^n}$ is Noetherian for $n\gg 0$, it is complicated to find bounds on its lowest degrees. However, we are able to prove that there is a  polynomial lower bound, regardless of the characteristic of $k$. The proof rests on a result by Chardin, Ha, and Hoa (\cite{CHH}), and is provided in Section \ref{polySec}.  

In Section \ref{monoSec} we  focus on the case where $I$ is a monomial ideal in a polynomial ring $R$. As expected, the extra combinatorial structure allows for better results. Assuming that $\HH{i}{\fm}{R/I^n}$ is Noetherian for $n\gg 0$, one can show that either $\indeg\HH{i}{\fm}{R/I^n}=0$ for $n\gg 0$ or  $\displaystyle\liminf_{n\rightarrow \infty} \frac{ \indeg\HH{i}{\fm}{R/I^n}}{n}\gs 1$, and the latter holds precisely when $\tilde H_{i-1}(\Delta(I))= 0$, where $\Delta(I)$ is the simplicial complex whose Stanley-Reisner ideal is  $\sqrt{I}$.

\section{Symmetric Powers of Locally Free Modules and Linear Lower Bound}\label{linearSec}
Let $E$ be a Noetherian graded module  and set  $u=\mu(E):=\dim_k E/\fm E$. Let  $$F_1\xrightarrow{\phi} F_0\to E\to 0$$
be a minimal presentation of $E$,  where $F_0$ and $F_1$ are graded free $R$-modules, and $\phi$ is an $u\times s$ matrix with entries in $\fm$. Let $T_1,\ldots, T_u$ be a set of variables and $\ell_1,\ldots, \ell_s$ the linear forms determined by 
$$\begin{bmatrix}\ell_1,&\cdots,&\ell_s
\end{bmatrix}=
\begin{bmatrix}T_1,&\cdots,&T_u
\end{bmatrix}\phi.
$$
The ring  $\Sym(E):=R[T_1,\ldots, T_u]/(\ell_1,\cdots, \ell_s)$ is  the {\it symmetric algebra} of $E$. Let $d_1,\ldots, d_u$ be the degrees of a homogeneous minimal generating set of $E$.   We can assign to $\Sym(E)$ a bi-graded structure where $T_i$ has bi-degree $(d_i,1)$ for every $i=1,\ldots, u$. The $n$th-graded component of $\Sym(E)$,  $S^n(E)=\oplus_{a\in \ZZ}\Sym(E)_{(a,n)}$, is the {\it $n$th-symmetric power} of $E$. 
%$$\Sym(E)\rightarrow \Sym(F)=R[T_1,\ldots,T_\gamma],$$
%Let $F=\oplus_{i=1}^\gamma R(-d_i)$ be a free $R$-module for some non-negative integers $d_i$. 
%Naturally, $F$ is a graded $R$-module, where  $$F_n:=\bigoplus_{n_1+\ldots +n_\gamma=n}^\gamma R(-d_i)_{n_i}$$ for every $n\in \ZZ$.  
%Let $E\subset F$ be a graded submodule, then we obtain a map of {\it symmetric algebras} $$\Sym(E)\rightarrow \Sym(F)=R[T_1,\ldots,T_\gamma],$$ where each $T_i$ has bidegree $(d_i, 1)$. The image of this map is the bi-graded algebra$$\R[E]:=\oplus_{n\gs 0}E^n\subset R[T_1,\ldots,T_\gamma].$$ The ring $\R[E]$ is called the {\it Rees algebra} of $E$ with respect to the embedding $E\subset F$. 
%It is known that if $E$ has a {\it rank}, i.e., $E_P$ is free of constant rank for every $P\in \Ass(R)$, then $\R[E]$ is isomorphic fo $\Sym(E)/(\text{$R$-torsion})$ and hence it is independent of the graded embedding of $E$ into a free module (cf. \cite{EHU}).
%Let $G_n:=\Sym(F)_n=\oplus_i R(-g_i)$ then $E^n$ is a graded submodule of $G_n$.  
%For any graded $R$-module $M$, we define the product $E^nM$ as the image of the following composition of maps 
%$$E^n\otimes_R M\rightarrow G_n\otimes_R M\xrightarrow{\sim} \oplus_i M(-g_i).$$ 

Let $M$ be any Noetherian graded $R$-module and $U\subseteq E$ a graded submodule. We say $U$ is an {\it $M$-reduction} of $E$, if $S^n(E)\otimes_R M=S^1(U)S^{n-1}(E)\otimes_RM$ for $n\gg 0$, where $S^1(U)$ is seen as a submodule of $S^1(E)$. Following \cite{TW}, we define $$\rho_M(E):=\min\{\beta(U)\mid U \text{ is an $M$-reduction of } E \}.$$ We note that  $\rho_M(E)\ls \beta(E)$ for every $R$-module $M$. The following theorem is the module version of \cite[3.2]{TW} and the proofs of these results are identical, however we include some relevant details for the reader's convenience. We remark that even though the algebras in \cite{TW} are positively graded, the proof of this result does not use this assumption.

\begin{thm} \label{regModules}
Let $R$ be a standard graded algebra over a Noetherian ring $A$. Let $E$ and $M$ be finitely generated graded $R$-modules. Then $$\reg(S^n(E)\otimes_R M)=\rho_M(E)n+e$$
for some integer $e\gs \indeg M$ and every $n\gg 0$.
\end{thm}
\begin{proof}

Let $U$ be an $M$-reduction of $E$ such that $d(U)=\rho_M(E)$. Let $\M=\Sym(E)\otimes_R M = \oplus_{n\in \NN} S^n(E)\otimes_R M$ and notice $\M$ is a finitely generated graded $\Sym(U)$-module. Let $s=\mu_A(R_1)$, $v=\mu(U)$, and $u_1,\ldots, u_v$ the degrees of a homogeneous minimal generating set of $U$, then $\Sym(U)$ is a quotient ring of the bi-graded polynomial ring $A[x_1,\ldots, x_s, y_1,\ldots, y_v]$ where $x_i$ has degree $(1,0)$  for each $i$ and $y_j$ has degree $(u_i,1)$ for each $j$. 
Therefore, \cite[2.2]{TW} implies $\reg(S^n(E)\otimes_R M)$ is a linear function $\rho n+e$ for some $e$ and $\rho\ls \rho_M(E)$. Finally, proceeding as in \cite[3.1]{TW} we obtain $\rho \gs \rho_M(E)$ and $e\gs \indeg M$, finishing the proof. 
\end{proof}

For the next result, we assume $R$ is a local ring or standard graded over a field. Let $\fm$ be the (irrelevant) maximal ideal of $R$,  $k=R/\fm$, and $E_R(k)$ the {\it (graded)  injective hull} of $k$. For a (graded) $R$-module $M$ we set $$M^\vee:=\Hom_R(M,\, E_R(k)).$$

The following is a generalization of a duality result of Horrocks (\cite{Hor}).

\begin{prop}\label{duality}
Let $(R,\fm,k)$ be a Cohen-Macaulay local ring $($or positively graded $k$-algebra$)$. Set
%$d=\dim R\gs 3$. 
$d=\dim R\gs  2$ 
and $\omega$ a (graded) canonical module of $R$. Fix $1\ls i\ls d-1$, then for a $($graded$)$ $R$-module $M$ of dimension $d$ that is $S_{i+1}$ locally on $\Spec R\setminus \{\fm\}$ we have $($graded$)$ isomorphisms
$$\HH{i}{\fm}{M}^\vee\cong \HH{d-i+1}{\fm}{\Hom_R(M,\omega)}\quad\text{if } i\gs 2,$$
and,
$$\HH{1}{\fm}{M}^\vee\cong \ker\big(\HH{d}{\fm}{\Hom_R(M,\omega)}\rightarrow \HH{d}{\fm}{\Hom_R(F_0,\omega)}\big),$$
where $F_0\twoheadrightarrow M$ is a free module.
\end{prop}

\begin{proof} We begin the proof with the following claim. 

\

\noindent {\it {\bf Claim:} Let $N$ be a $($graded$)$ $R$-module that is Maximal Cohen-Macaulay $($MCM$)$ locally on $\Spec R\setminus \{\fm\}$ and  
$\cdots\to  F_{ 0} \rightarrow N\to 0$ a $($graded$)$ free resolution of $N$. Then,  $\Ext^1_R(N, \omega)\cong \ker\big(\HH{2}{\fm}{\Hom_R(N,\omega)}\rightarrow \HH{2}{\fm}{\Hom_R(F_0,\omega)}\big)$ if $d=2$, and $\Ext^1_R(N, \omega)\cong \HH{2}{\fm}{\Hom_R(N, \omega)}$ if $d\gs 3$.}

In order to prove this claim, we consider the $R$-modules $K$ and $C$ that fit in the following two exact sequences 
\begin{equation}\label{qq1}
 0\rightarrow\Hom_R(N,\omega)\rightarrow \Hom_R(F_0,\omega)\rightarrow C\rightarrow 0 \quad \quad\text{and}
 \end{equation}
 \begin{equation}\label{qq2}
0\rightarrow K\rightarrow \Hom_R(F_1,\omega)\rightarrow\Hom_R(F_2,\omega).
\end{equation}
By applying the depth lemma to \eqref{qq2} we obtain $\depth K\gs 2$, therefore $$\Ext^1_R(N,\omega)=\HH{0}{\fm}{\Ext^1_R(N,\omega)}=\HH{0}{\fm}{K/C}\cong \HH{1}{\fm}{C}$$ where the first equality follows by the assumption on $N$.  Hence, the conclusion of the claim  follows  from \eqref{qq1}.

\

Now, back to the original statement, we note that the result follows by the claim and local duality \cite[3.6.19]{BH} if $d=2$, then we may assume $d\gs 3$. Let $\Omega^nM$ be the $n$th-syzygy module of $M$.  Again by local duality and the claim we have 
\begin{equation}\label{rr1}
\HH{i}{\fm}{M}^\vee\cong \Ext_R^{d-i}(M,\omega)\cong \Ext^1_R(\Omega^{d-i-1}M,\omega)\cong \HH{2}{\fm}{\Hom_R(\Omega^{d-i-1}M,\omega)}.
\end{equation}
%hence it suffices to prove $ \HH{2}{\fm}{\Hom_R(\Omega^{d-i-1}M,\omega)}\cong \HH{d-i+1}{\fm}{\Hom_R(M,\omega)}$.
Let $0\ls t\ls d-i-2$, by assumption $\Omega^tM$ is MCM in codimension $i+t+1$,  which implies that $\dim \Ext^1_R(\Omega^{t}M,\omega)<d-i-t-1$. Let $\cdots\to  F_{ 0} \rightarrow M\to 0$ be a (graded) free resolution of $M$. From the exact sequence 
$$0\rightarrow\Hom_R(\Omega^{t}M,\omega)\rightarrow\Hom_R(F_t,\omega)\rightarrow\Hom_R(\Omega^{t+1}M,\omega)\rightarrow \Ext^1_R(\Omega^{t}M,\omega)\rightarrow 0$$
we obtain 
\begin{equation}\label{rr2}
\HH{d-i-t}{\fm}{\Hom_R(\Omega^{t+1}M,\omega)}\cong \HH{d-i-t+1}{\fm}{\Hom_R(\Omega^{t}M,\omega)},\,\,\,\,\,\text{ if } i+t\gs 2,
\end{equation}
and
\begin{equation}\label{rr3}
\HH{d-1}{\fm}{\Hom_R(\Omega^{1}M,\omega)}\cong \ker\big(\HH{d}{\fm}{\Hom_R(M,\omega)}\rightarrow \HH{d}{\fm}{\Hom_R(F_0,\omega)}\big).
\end{equation}
The statement now follows from \eqref{rr1}, \eqref{rr2}, and \eqref{rr3}.
\end{proof}

Given an $R$-module $M$, we denote by $\Gamma(M)$ the {\it divided powers algebra} of $M$  \cite[Appendix 2]{E}. We set $M^*:=\Hom_R(M,R)$ and for a graded $R$-algebra $S=\oplus_{n\in \NN}S_n$, we denote by $S^*:=\oplus_{n\in \NN}S_n^*$ the {\it graded dual} of $S$.

We need the following technical lemma for the proof of our main result.

\begin{lemma}\label{symm}
Let $R$ be a commutative $($graded$)$ ring and $M$ a $($graded$)$ $R$-module. Then there exist natural $($graded$)$ maps 
$$\Sym(M^*)\xrightarrow{\alpha} \Gamma(M^*)\xrightarrow{\beta} \Sym(M)^*.$$
Moreover, $\alpha$ is an isomorphism if $R$ contains the field of rational numbers and $\beta$ is an isomorphism if $M$ is free. 
\end{lemma}
\begin{proof}
For the construction and results on $\alpha$, see \cite[Proposition III.3., page 256]{Roby}. See \cite[A2.6 and A2.7(c)]{E} for the corresponding information for $\beta$.
\end{proof}

The following is the main theorem of this section. %We recall that if $d=\dim R$, the {\it $a$-invariant} of $R$ is defined as $a(R)=\topdeg \HH{d}{\fm}{R}$. If $R$ is Cohen-Macaulay with a canonical module $\omega$ then we also have  $a(R)=-\indeg \omega$.
 
\begin{thm}\label{mainVB}
Let $(R,\fm)$ be a Cohen-Macaulay standard
graded algebra over a field $k$ of characteristic zero. Set $d=\dim R\gs  2$. Let $E$ be a  graded $R$-module which is free locally on $\Spec R\setminus \{\fm\}$.
% Assume $E$ satisfies one of the following conditions:
%\begin{enumerate}
%\item[$(i)$] $E$ is locally free in the punctured spectrum.
%\item[$(ii)$] $E$ is MCM and has finite projective dimension in the punctured %spectrum. Moreover, $E$ is reflexive in codimension one $($e.g. $R_P$ is %Gorenstein for very $\height P\ls 1$ $)$.
%\end{enumerate}
Then there exists an integer $\varepsilon$ such that $$\indeg \HH{i}{\fm}{S^n(E)} \gs -\beta(E^*) n+\varepsilon$$ for every $n\gs 1$ and $ 1\ls i< d$.
\end{thm}
\begin{proof} %The statement is trivial for $i=0$ and $ \HH{1}{\fm}{S^n(E)}\cong  \HH{0}{\fm}{S^n(F)/S^n(E)}$, then we may assume $d\gs 3$ and $i\gs 2$.

First, assume $i\gs 2$. Let $\omega$ be the canonical module of $R$.
%, $M^\omega:=\Hom_R(M,\omega)$, and   $M^\vee=\Hom_R(M,E_R(k))$.
%Let $F_1$ be a free module of rank $\mu(E)$ and $\phi: F_1\rightarrow F$ a homogeneous map such that $\Image(\phi)=E$, hence $E^*\hookrightarrow F_1^*$. The latter embedding can be used to compute $\R[E^*]$, however, notice that by assumption $E$ (hence $E^*$) has rank $\gamma$ , then $\R[E]$ and $\R[E^*]$ are independent of the embedding to a free module. 
By Hom-Tensor adjointness and the isomorphism $R\cong \Hom_R(\omega,\, \omega)$, we have $S^n(E^*)^*\cong \Hom_R(S^n(E^*)\otimes_R \omega, \omega).$ 

%We claim that $\Hom_R((E^*)^n\otimes_R \omega,\, \omega)\cong \Hom_R((E^*)^n\omega,\, \omega)$ for which it is enough to show that the kernel of $(E^*)^n\otimes_R \omega\rightarrow (E^*)^n\omega$ is a torsion $R$-module. Indeed, for every $P\in \Ass(R)$ we have by assumption that  $E_P$ is free, and hence a direct summand of $(F_1)_P$. 
%This implies $E^*_P$ is a direct summand of $(F_1)_P^*$ and then $(E^*)^n_P$ is a direct summand of $\Sym((F_1)_P^*)_n$. Therefore,
  %$\big((E^*)^n\otimes_R \omega\big)_\fp\rightarrow \big((E^*)^n\omega\big)_\fp$ is injective and the claim follows. This argument leads to the following isomorphism
 %\begin{equation}\label{firstStep}
%$$(( E^*)^n)^*\cong \Hom_R((E^*)^n \omega,\, \omega).$$
% \end{equation}
 %and  set $T_n:=(E^*)^n \omega$, 
 By the assumption we have $S^n(E^*)\otimes_R \omega$ is MCM locally on $\Spec R\setminus \{\fm\}$, therefore the natural map $$S^n(E^*)\otimes_R \omega\rightarrow  \Hom_R(\Hom_R( S^n(E^*)\otimes_R \omega,\, \omega),\, \omega)$$
 is an isomorphism locally on $\Spec R\setminus \{\fm\}$. Hence,  by Proposition \ref{duality}
$$
\HH{i}{\fm}{S^n( E^*)^*}\cong \HH{i}{\fm}{\Hom_R( S^n(E^*)\otimes_R \omega,\, \omega)}\cong \HH{d-i+1}{\fm}{S^n(E^*)\otimes_R \omega }^\vee,
$$
%if $i\gs 2$, and $\HH{1}{\fm}{S^n( E^*)^*}\cong \N^\vee$ for some submodule $\N$ of $\HH{d}{\fm}{S^n(E^*)\otimes_R \omega }$. 
By  Theorem \ref{regModules}, we have $$\topdeg \HH{d-i+1}{\fm}{S^n(E^*) \otimes_R \omega }\ls \beta(E^*) n-\varepsilon$$ for some $\varepsilon \in \ZZ$ and every $n\gs 1$. Therefore $ \indeg \HH{i}{\fm}{S^n( E^*)^*}\gs -\beta(E^*) n+\varepsilon$ for every $n\gs 1$ and $i\gs 2$. % Now, since $E$ and $E^*$ have a rank, we have $E^n\cong \Sym(E)_n/(R\text{-torsion})$ and $( E^*)^n\cong\Sym(E^*)_n/(R\text{-torsion})$, therefore 
%\begin{equation}\label{ee2}(( E^*)^n)^*\cong (\Sym(E^*)_n)^* \qquad  \text{and} \qquad (( E^n)^*)^*\cong ((\Sym(E)_n)^*)^*.
%\end{equation} 
  The map $\Sym(E^*)\xrightarrow{\beta\circ\alpha} \Sym(E)^*$ in Lemma \ref{symm} (with $M=E$) is an  isomorphism locally on $\Spec R\setminus \{\fm\}$, hence 
\begin{equation}\label{isomSym}
S^n( E^*)^*\cong S^n( E)^{**}. 
\end{equation}
The result now follows for $i\gs 2$ by observing that $\HH{i}{\fm}{S^n( E)}\cong \HH{i}{\fm}{S^n( E)^{**}}
$ as $S^n(E)$ is free, hence reflexive, locally on $\Spec R\setminus \{\fm\}$.

Now, we show the statement for $i=1$. Fix $n\gg 0$ and consider the short exact sequence $$0\to \Ker(\varphi) \to S^n(E)\xrightarrow{\varphi} S^n(E)^{**}\to \C\to 0.$$
Since $\varphi$ is an isomorphism locally on $\Spec R\setminus \{\fm\}$, we have $\dim \Ker (\varphi) = \dim \C = 0$. Then $\HH{1}{\fm}{S^n(E)}\cong\C$, as $\depth S^n( E)^{**}\gs 2$. Therefore, $$\indeg \HH{1}{\fm}{S^n(E)} = \indeg \C \gs \indeg S^n(E)^{**}=\indeg S^n( E^*)^*,$$
where the last equality follows from \eqref{isomSym}.  Let $\oplus_{i=1}^u R(-a_i)\rightarrow S^n( E^*)\rightarrow 0$ be the first map of a minimal homogeneous resolution of $S^n( E^*)$, where $u=\mu(S^n( E^*))$. Then, $S^n( E^*)^*\hookrightarrow \oplus_{i=1}^u R(a_i)$. We conclude $$\indeg S^n( E^*)^*\gs -\max_i\{a_i\}\gs -\reg (S^n( E^*))\gs -\beta(E^*)n+\varepsilon,$$
for some $\varepsilon\in \ZZ$ and $n\gs 1$ by Theorem \ref{regModules}.
\end{proof}

Assume  $E$ is a graded submodule of a free graded $R$-module $F=\oplus_{i=1}^\gamma R(-d_i)$. We have the natural map of symmetric algebras 
$$\Sym(E)\rightarrow \Sym(F)=R[T_1,\ldots,T_\gamma],$$ where each $T_i$ has bidegree $(d_i, 1)$. The image of this map is the bi-graded algebra$$\R[E]:=\oplus_{n\in \NN}E^n\subset R[T_1,\ldots,T_\gamma].$$ The ring $\R[E]$ is called the {\it Rees algebra} of $E$ with respect to the embedding $E\subset F$. 
It is known that if $E$ has a {\it rank}, i.e., $E_P$ is free of constant rank for every $P\in \Ass(R)$, then $\R[E]$ is isomorphic fo $\Sym(E)/(\text{$R$-torsion})$ and hence it is independent of the graded embedding of $E$ into a free module (\cite{EHU}).

\begin{cor}
Let $(R,\fm,k)$ and $E$ be as in Theorem \ref{mainVB}.  Assume that $E$   has  a rank, then  $$\indeg \HH{i}{\fm}{E^n} \gs  -\beta(E^*)n+\varepsilon$$
for some $\varepsilon\in \ZZ$ and every $n\gs 1$, $1\ls  i< d$.
\end{cor}
\begin{proof}
Since $E$ has a rank, and is free locally on $\Spec R  \setminus  \{\fm\}$, for every $n\gs 1$ we have  $E^n\cong S^n(E)/(R\text{-torsion})$ and this $R$-torsion submodule is supported on $\{\fm\}$.
 Therefore,
 $\HH{i}{\fm}{E^n}=\HH{i}{\fm}{S^n(E)}$ for every $i\gs 1$   and the statement follows from Theorem \ref{mainVB}.
\end{proof}

\begin{cor}\label{conormal}
Let $(R,\fm)$ be a standard graded algebra over a field $k$ of characteristic zero. Let $I$ be a homogeneous $R$-ideal such that $S=R/I$  is Cohen-Macaulay. Assume $I_\fp$ is generated by a regular sequence in $R_\fp$  for every $\fp\in \Spec(R)\setminus \{\fm\}$  and that $\dim S\gs 2$. Let $E=I/I^2$ be the {\it conormal module} of $I$ and $E^*=\Hom_{S}(E, S)$, then

\begin{enumerate}
%\item if $\beta(E^*)<0$ then for each  $1\ls  i< \dim R/I$ the sequence $\{\indeg \HH{i}{\fm}{R/I^n}\}_{n\in \NN}$ is constant for $n\gg 0$;

\item if $\beta(E^*)\ls0$,  there exists $C\in \ZZ$ such that $$\indeg \HH{i}{\fm}{R/I^n} \gs C$$
for every $n\gs 1$ and $1\ls  i< \dim R/I$;

\item if $\beta(E^*)>0$, then there exists $\varepsilon\in \ZZ$ such that $$\indeg \HH{i}{\fm}{R/I^n} \gs -\beta(E^*)n+\varepsilon$$
for every $n\gs 1$ and $1\ls  i< \dim R/I$.
\end{enumerate}
 
\end{cor}
\begin{proof}
By assumption $ E$ is  a $R/I$-module that is free locally on $\Spec R\setminus \{\fm\}$. Since the natural epimorphism $S^n(E)\twoheadrightarrow I^n/I^{n+1}$ is an isomorphism locally on $\Spec R\setminus \{\fm\}$, we have $\HH{i}{\fm}{S^n(E)}\cong \HH{i}{\fm}{I^n/I^{n+1}}$ for $i\gs 1$. The conclusion now follows from Theorem \ref{mainVB} and by induction on $n$ via the inequality 
%$$0\rightarrow I^n/I^{n+1}\rightarrow R/I^{n+1}\rightarrow R/I^n\rightarrow 0$$ we have 
$$\indeg \HH{i}{\fm}{R/I^{n+1}} \gs \min\{\indeg \HH{i}{\fm}{I^n/I^{n+1}},\,\indeg \HH{i}{\fm}{R/I^{n}}\}$$ for $n\gs 1$.
\end{proof}

\begin{example}\label{limKnown}
Let $R=k[x,y,u,v]/(xu^t-yv^t)$ for some $t\gs 1$ and $\chara k = 0$. Let  $I=(x,y)$ and notice that $I$ and the ring $S=R/I$ satisfy the assumptions of Corollary \ref{conormal}. The graded free resolution of $I/I^2$ is
$$0\rightarrow S(-1-t) 
\xrightarrow{
\begin{bmatrix} u^t\\-v^t\end{bmatrix}
}
S^2(-1)
\rightarrow
 I/I^2\rightarrow 0.$$
 Therefore,  $(I/I^2)^*$ is isomorphic to the kernel of the map 
$S^2(1) 
\xrightarrow{
\begin{bmatrix} u^t&-v^t\end{bmatrix}
}
S(1+t)
$, which is generated by $\begin{bmatrix} v^t\\u^t\end{bmatrix}$. Hence $\beta((I/I^2)^*)=t-1$.

If $t\gs 2$, by Corollary \ref{conormal} we have $\indeg \HH{1}{\fm}{R/I^n} \gs  -(t-1)n+\varepsilon$  for some  $\varepsilon\in\ZZ$ and every $n\gs 1$. On the other hand, computing $ \HH{1}{\fm}{R/I^n}$ via the \u{C}ech complex of the system of parameters $\{u,v\}$ of the ring $R/I^n$, we obtain that the class of $[ \frac{x^{n-1}}{v^{tn-t}} , \, \frac{y^{n-1}}{u^{tn-t}} ]$ is nonzero and has degree $-(t-1)n+(t-1)$. We conclude that $$-(t-1)n+(t-1)\gs \indeg \HH{1}{\fm}{R/I^n} \gs -(t-1)n+\varepsilon$$ for every $n\gs 1$. Therefore, $$\lim_{n\rightarrow \infty}\frac{ \indeg \HH{1}{\fm}{R/I^n}}{n}=-(t-1).$$

Now, if $t=1$ we have $\beta((I/I^2)^*)=0$ and hence $\{\indeg \HH{1}{\fm}{R/I^n}\}_{n\in \NN}$ is bounded below by a constant.  

In fact, computations with Macaulay2 \cite{GS} suggests that the sequence $\{\indeg \HH{1}{\fm}{R/I^n}\}_{n\in \NN}$ in Example \ref{limKnown} agrees with the linear function $-(t-1)(n-1)$ for each $t\gs 1$ and   $n\gs 2$. We record some of these values below.

\begin{table}[ht]
%\caption{Some values of \indeg \HH{1}{\fm}{R/I^n}$.}
\centering
\begin{tabular}{c| c c c c c c c c c c c c c c c c}
%\hline
%\hline
\backslashbox{$t$}{$n$} & 2 & 3 & 4 & 5 & 6 & 7 & 8 & 9 & 10 & 11 & 12 & 13 & 14 & 15 & 16   \\
 \hline
1 & 0 & 0 & 0  & 0 & 0 & 0 & 0 & 0 & 0 & 0 & 0 & 0 & 0 & 0 & 0  \\
 2 & -1 & -2 & -3 & -4 & -5 & -6 & -7 & -8 & -9 & -10 & -11 & -12 & -13 & -14 & -15 \\
  3 &   -2 & -4 & -6 & -8 & -10 & -12 & -14 & -16 & -18 & -20 & -22 & -24 & -26 & -28 & -30 \\
   4 &  -3 & -6 & -9 & -12 & -15 & -18 & -21 & -24 & -27 & -30 & -33 & -36 & -39 & -42 & -45  \\
     5 &   -4 & -8 & -12 & -16 & -20 & -24 & -28 & -32 & -36 & -40 & -44 & -48 & -52 & -56 & -60 \\
\end{tabular}
\label{sequences}
\end{table}
\end{example}

\begin{example}\label{maxMinors}
Let $X$ be a $2\times 3$ generic matrix and $ R=k[X]$  with  $\chara k = 0$. Let $I=I_2(X)$ the ideal generated by the $2\times 2$ minors of $X$, then $R/I$ is  Cohen-Macaulay of dimension 4, and $\Proj R/I$ is lci.  
Using Macaulay2 \cite{GS} we obtain $\beta((I/I^2)^*) =-1$, therefore by Corollary \ref{conormal} the sequence $\{\indeg \HH{3}{\fm}{R/I^n}\}_{n\in \NN}$ is bounded below by a constant. Indeed, $\indeg \HH{3}{\fm}{R/I^n} = 0 $ for every $n\gs 2$ \cite[5.1]{BBLSZ}.
\end{example}

In the following example we demonstrate  that the lower bound $C$ in Corollary \ref{conormal} (1) may be negative.

\begin{example}
Let $R=k[x,y,z,u,v,w]/(x^2u^2+y^2v^2+z^2w^2)$ with  $\chara k = 0$ and  let  $I=(x,y,z)$. Computations with Macaulay2 \cite{GS} show that $\beta((I/I^2)^*) =-1$ and suggest that $\indeg \HH{i}{\fm}{R/I^n}=-2$ for every $n \gs 3$.
\end{example}

In the following example we observe that, even when the ring $R$ is regular, the sequence $\{\indeg \HH{i}{\fm}{R/I^n}\}_{n\in \NN}$ may have linear behavior  with negative slope.

\begin{example}
Let $R=k[x,y,u,v]$ with $\chara k = 0$ and $I = (x^2u-y^2v,u^2,uv,v^2)$. Computations with Macaulay2 \cite{GS} show that $\beta((I/I^2)^*) =2$ and suggest that $\indeg \HH{i}{\fm}{R/I^n}=-2n+1$ for every $n \gs 1$.
\end{example}

In \cite[6.1]{EMS}, Eisenbud, Musta\c{t}\u{a}, and Stillman asked for which ideals $I$ in a polynomial ring $R$ over a field $k$, there exists a decreasing chain of ideals $\{I_n\}_{n\in \NN}$, cofinal with the regular powers $\{I^n\}_{n\in \NN}$, such that the natural map 
$$\Ext_R^i(R/I_n,R)\rightarrow
\lim_{\longrightarrow}\Ext_R^i(R/I_n,R)=\HH{i}{I}{R}$$
is injective for every $i$ and $n$. In \cite[1.2]{BBLSZ} the authors provide a partial answer to this question, showing that if $I$ is a homogeneous prime ideal such that $\Proj R/I$ is smooth, and $\chara k = 0$, then for each $i\in \NN$ and $j\in \ZZ$ the map $$\Ext_R^i(R/I^n,R)_j\rightarrow
\HH{i}{I}{R}_j$$
is injective for  $n\gg 0$. Part  (3) of the following corollary provides another partial answer to the  question above for a class of ideals in Gorenstein rings. Part (1) and (2) are closely related to \cite[1.1]{BBLSZ}.

\begin{cor}\label{stabExt}
Let $R$, $I$, and $E$ be as in Corollary \ref{conormal}. Assume $\dim R/I\gs 3$  and $\beta(E^*)<0$, then for any  $D\in \ZZ$, we have
\begin{enumerate}
\item for each $1\ls i\ls  \dim R/I-2$  the natural map  $$\HH{i}{\fm}{R/I^{n+1}}_{\ls D}\rightarrow\HH{i}{\fm}{R/I^{n}}_{\ls D}$$ is an isomorphism for $n\gg 0$; this map is injective for $i=\dim R/I -1$ and $n\gg 0$. 

\item If $R$ is Cohen-Macaulay and $\omega$ is the canonical module of $R$, then for each $\height I +2\ls i<\dim R$ the natural map
$$\Ext_R^i(R/I^n,\omega)_{\gs D}\rightarrow \Ext_R^i(R/I^{n+1},\omega)_{\gs D}$$
is an isomorphism for $n\gg 0$. Furthermore,  this map is injective for $i=\height I$ and $n\gs 1$.

\item  If $R$ is Cohen-Macaulay and $\omega$ is the canonical module of $R$, then for every $i<\dim R$ such that $i\neq \height I +1 $ the natural map
$$\Ext_R^i(R/I^n,\omega)_{\gs D}\rightarrow \HH{i}{I}{\omega}_{\gs D}$$
is injective for $n\gg 0$. In fact the map is injective for $i=\height I$ and $n\gs 1$, and it is an isomorphism for $i\gs \height I+2$ and $n\gg 0$.
\end{enumerate}
\end{cor}
\begin{proof}
Part (1) follows by Theorem \ref{mainVB}, as the assumptions imply that $\indeg \HH{i}{\fm}{I^n/I^{n+1}} = \indeg \HH{i}{\fm}{S^n(E)}>D$ for $1\ls i\ls \dim R/I-1$ and $n\gg 0$. Part (2) follows from (1) and local duality for the case $i\gs\height I+2$. The injectivity for $i=\height I$ follows by local duality and the epimorphism $$\HH{\dim R/I}{\fm}{R/I^{n+1}}\twoheadrightarrow \HH{\dim R/I}{\fm}{R/I^{n}}$$ for $n\gs 1$.
Now, Part (3) follows from (2) as $\HH{i}{I}{\omega}=\displaystyle\lim_{\longrightarrow}\Ext_R^i(R/I^n,\omega)$ for every $i$.
\end{proof}

A local ring $(S,\fn)$ is said to be {\it cohomologically full} if for every surjection $T\twoheadrightarrow S$ from a local ring $(T,\ffq)$, such that $T_{\text{red}} = S_{\text{red}}$ and $T$ and $S$ have the same characteristic, the natural map 
$\HH{i}{\ffq}{T}\rightarrow \HH{i}{\fn}{S}$
is surjective for every $i$. If $R$ is a standard graded algebra over a field $k$ and irrelevant maximal ideal $\fm$, then we say $R$ is  cohomologically full if the local ring $R_\fm$ is.   For more information and examples of cohomologically full rings see \cite{DDM}.
\vspace{1mm}

The following result answers Question \ref{motivQ}, (2), in a particular case.

\begin{cor}\label{cohFull}
Let $R$, $I$, and $E$ be as in Corollary \ref{stabExt}, and fix an integer $1\ls i\ls \dim R/I-2$.  Assume $R/J$  is cohomologically full for some $R$-ideal $J$ such that $\sqrt{J}=\sqrt{I}$ and $\HH{i}{\fm}{R/J}\neq 0$. Then there exists an integer $C\ls 0$ such that $\indeg \HH{i}{\fm}{R/I^n}=C$ for every $n\gg 0$.
\end{cor}
\begin{proof}
 By assumption we have that the map $\HH{i}{\fm}{R/I^n}\rightarrow \HH{i}{\fm}{R/J}$ is surjective for $n\gg 0$. Therefore,  $ \HH{i}{\fm}{R/J}$ has finite length and  
$$\indeg \HH{i}{\fm}{R/I^n}\ls \indeg \HH{i}{\fm}{R/J} =0,$$
where the last equality follows from \cite[4.9]{DDM}. Now, by Corollary \ref{stabExt}, (1) we have that  $\HH{i}{\fm}{R/I^{n+1}}_{\ls 0}\rightarrow\HH{i}{\fm}{R/I^{n}}_{\ls 0}$ is an isomorphism for $n\gg 0$. The conclusion follows.
\end{proof}

\begin{remark}
In the setting of Corollary \ref{cohFull}, let $X = \Proj R/I$. If $i\ls \codim \Sing X$ it was proved in \cite[3.1]{BBLSZ} that $\HH{i}{\fm}{R/I^n}_{<0}=0$ for every $n\gs 1$. 
Hence, if $I^n$ is  cohomologically full for every $n\gg 0$, \cite[4.9]{DDM} shows $\indeg \HH{i}{\fm}{R/I^n}=0$ for $n\gg 0$.
\end{remark}

The following example shows that the assumption on the characteristic  is necessary in Corollary \ref{conormal} and hence in Theorem \ref{mainVB}.

\begin{example}
Let $R$ and $I$ be  as in Example \ref{maxMinors} but assume instead that $R$ has characteristic $p>0$. Moreover, assume the conclusion of Theorem \ref{mainVB} holds in positive characteristic. Computations by Macaulay2 \cite{GS} show $\beta((I/I^2)^*) =-1$ and then by Corollary \ref{stabExt}, (1) we have $$\HH{3}{\fm}{R/I^{n+1}}_0\rightarrow \HH{3}{\fm}{R/I^{n}}_0$$
is injective for $n\gg 0$. Furthermore, by \cite[5.5]{BBLSZ}, $\HH{3}{\fm}{R/I^{n}}_0\neq 0$ for every $n\gs 2$. However, if $n'\gg n\gg0$, there exists $e\in \NN$ such that $I^{n'}\subseteq I^{[p^e]}\subseteq I^n$ and hence $\HH{3}{\fm}{R/I^{n'}}\rightarrow \HH{3}{\fm}{R/I^{n}}$ is the zero map as $R/I^{[p^e]}$ is Cohen-Macaulay, which is a contradiction. 
\end{example}

\section{Polynomial Bound for Homogeneous Ideals}\label{polySec}

Let $I$ be a homogeneous ideal in a standard graded ring over a field. The purpose of this section is to prove that whenever the modules $ \HH{i}{\fm}{R/I^n}$ are Noetherian for $n\gg 0$, the rate of growth of the sequence $\{\indeg \HH{i}{\fm}{R/I^n}\}_{n\in \NN}$ is at most polynomial. The results of this section apply in wide generality and without assumptions on the characteristic of the base field. 

Let $M$ be a Noetherian $R$-module of dimension $d$. We denote by $e_0(M), \ldots, e_d(M)$ the {\it Hilbert coefficients} of $M$, i.e., 
$$\lambda(M/\fm^nM)=\sum_{i=0}^r (-1)^ie_i(M){n+d-i\choose d-i},\qquad \text{ for }n\gg 0,$$
where $\lambda(N)$ denotes the {\it length} of the $R$-module $N$.
\vspace{1mm}

%We say that a sequence of integers $\{a_n\}_{n\in\NN}$ is $O(n^s)$ if $\limsup_{n\rightarrow \infty}|\frac{a_n}{n^s}|<\infty$.

%\begin{lemma}
%Let $R$ be a Noetherian standard graded  $k$-algebra of dimension $d$ and let $\{I_n\}_{n\gs 1}$ be a system of homogeneous $R$-ideals. If $\reg(I_n)=O(n^t)$, then $\mu(I_n)= O(n^{td})$.
%\end{lemma}
%\begin{proof}
%First we note that for every $n\gg 0$ we have $\H_R(n):=\sum_{i=0}^n \dim_k R_i = O(n^d)$. Let $n\gg 0$ and $a_1^n,\ldots, a_{u_n}^n$ be the degrees of a minimal set of homogeneous generators of $I_n$, where $u_n=\mu(I_n)$. Set $A_n := \max_i\{a_i^n\}$, then $A_n =O(n^t)$ by assumption. We conclude that, $u_n\ls \sum_{i=0}^{A_n}\dim_k R_i = \H(A_n)=O(n^{td})$, as desired.
%\end{proof}

We now present the main theorem of this section.

\begin{thm}\label{polBound}
Let $R$ be a standard graded algebra over a field  $k$ and with irrelevant maximal ideal $\fm$. Let  $I$  be a homogeneous $R$-ideal and set $d=\dim R/I\gs 2$. Assume  $\HH{i}{\fm}{R/I^n}$ is Noetherian for some $1\ls i< d$ and $n\gg 0$. Then  there exists $s\in  \NN$ such that for every $n\gg 0$ we have
$$|\indeg \HH{i}{\fm}{R/I^n}|<n^s.$$
\end{thm}
\begin{proof}
Let $e=\dim_k R_1$. Consider an epimorphism $S := k[x_1,\ldots, x_e]\xtwoheadrightarrow{\varphi} R$ from a  polynomial ring $S$ and  let $\fn = (x_1,\ldots,x_e)\subseteq S$. %Set $K= \Ker \varphi$.   
By graded local duality \cite[3.6.19]{BH}  we have a graded isomorphism 
$$
 \HH{i}{\fm}{R/I^n}\cong \HH{i}{\fn}{R/I^n} \cong \Ext^{e-i}_S(R/I^n, S)^{\vee}(e)
$$
for every $n\in \NN$. Therefore, by the assumption the module $\Ext^{e-i}_S(R/I^n, S)$ has finite length for $n\gg 0$, and then
\begin{equation}\label{upsDo}
\indeg  \HH{i}{\fm}{R/I^n}  = -\topdeg  \Ext^{e-i}_S(R/I^n, S)-e= -\reg(  \Ext^{e-i}_S(R/I^n, S))-e.
\end{equation}
For an $R$-module $M$ of dimension $r$, we set $Q_M(n)= \sum_{i=0}^r |e_i(M)|{n+r-i\choose r-i}. $
 For an $R$-ideal $J$, set $\tilde{J}= (0:_{R}\fm^\infty)$.
\vspace{1mm}
 
  Let $r_n=\reg(R/I^n) $ and notice that from the regular sequence 
  $$0\to \tilde{I^n}/I^n\to R/I^n \to R/\tilde{I^n}\to 0$$ we obtain  $\reg(R/\tilde{I^n})\ls r_n $. Therefore,  by \cite[3.5]{CHH}, there exists  $C\in \NN$ such that  
$$\reg(  \Ext^{e-i}_S(R/I^n, S))<C\big(Q_{R/I^n}(r_n)\big)^{2^d-2}.$$
By \cite[1.1]{HPV}, the functions $e_i(R/I^n)$ agree with a polynomial of degree $\ls e-d-i$ for every $i$ and $n\gg 0$, therefore there exists a polynomial in two variables,  $q(n,t)\in \ZZ[n,t]$ of degree at most $e$ in $n$ and of degree $d$ in $t$,  such that $Q_{R/I^n}(t)\ls q(t,n)$ for $t, n\gg 0$. 
Since $r_n$ eventually agrees with a linear function by \cite[3.2]{TW},  it follows that $$\big(Q_{R/I^n}(r_n)\big)^{2^d-2} < D n^{(e+d)(2^d-2)}$$ for some $D\in \NN$ and $n\gg 0$. The conclusion now follows from \eqref{upsDo} and the fact that the sequence $\{\indeg \HH{i}{\fm}{R/I^n}\}_{n\in \NN}$ is bounded above by the linear function $\reg(R/I^n)$ for $n\gg 0$. 
\end{proof}

The previous result can be used to show a polynomial bound for the lengths of local cohomology modules of powers of homogeneous ideals. The next result  can be seen as a partial answer to Question  \cite[7.1]{DM}.

\begin{cor}\label{polynomialLength}
Let $R=k[x_1,\ldots, x_d]$, $\fm=(x_1,\ldots, x_d)$, and $I$  a homogeneous $R$-ideal. Assume $\HH{i}{\fm}{R/I^n}$ is Noetherian for $n\gg 0$, then there exists $t\in  \NN$ such that  $$\limsup_{n\rightarrow \infty} \frac{\lambda(\HH{i}{\fm}{R/I^n})}{n^t}<\infty.$$
\end{cor}
\begin{proof}
In this proof we follow Notation \ref{takNot}. By extending the field $k$ we can assume it is infinite.   As in the proof of \cite[5.3]{DM}, using   Gr\"obner deformation and \cite[2.4]{Sba} we can construct a sequence of monomial ideals $J_n$ such that  $\reg(J_n)=\reg(I^n)$ and $\dim_k \HH{i}{\fm}{R/I^n}_{j}\ls \dim_k \HH{i}{\fm}{R/J_n}_{j}$ for every $j\in\ZZ$. Let $\beta\in \NN$ be such that $\reg(I^n)\ls\beta n$ for every $n\gs 1$. By Theorem \ref{polBound}  there exists $s\in \ZZ_{>0}$ such that $\HH{i}{\fm}{R/I^n}_{<-n^s} = 0$, then 
\begin{equation}\label{beqn1}
\lambda(\HH{i}{\fm}{R/I^n})=\lambda(\HH{i}{\fm}{R/I^n}_{\gs -n^s})\ls\lambda(\HH{i}{\fm}{R/J_n}_{\gs -n^s})
\ls \sum_{-n^s \ls |\fa| \ls\beta n} \dim_k \HH{i}{\fm}{R/J_n}_{\fa}.
\end{equation} 
By \cite[5.2]{DM} if $|\fa^+|>\beta n$, then $\HH{i}{\fm}{R/J_n}_{\fa}=0$. Moreover, if $\fa=(a_1,\ldots, a_d)\in \ZZ^d$ satisfies that $-n^s \ls |\fa | \ls\beta n$ and $|\fa^+|\ls \beta n$, then for every  $j\in G_\fa$ we must have $a_j\gs -n^s-\beta n$. Set $S=G_\fa$, then by \cite[3.3, 5.2, 5.1]{DM} there exists $C\in \NN$ such that
\begin{equation}\label{beqn2}
\sum_{-n^s \ls |\fa| \ls\beta n,\, G_\fa =S}  \dim_k \HH{i}{\fm}{R/J_n}_{\fa}\ls (n^s+\beta n)^{|S|}Cn^{d-|S|}\ls  C(n^{s+1}+\beta n^2)^{d}.
\end{equation}
The conclusion now follows  from \eqref{beqn1}   and Inequality \eqref{beqn2},  by adding the latter over all possible $S$ and setting $t=d(s+1)$.
\end{proof}

We finish the section with the following remark.

\begin{remark}
Assume  $I$ is generated in a single degree $\gamma$. In \cite[3.4]{BCH} the authors showed that for every $i\gs 0$, there exists a set $\Lambda_i\subseteq \ZZ$ and a function $\eta_i:\Lambda_i\to \NN$, such that $\HH{i}{\fm}{R/I^n}_{l+n\gamma }\neq 0$ for every $l\in \Lambda_i$ and $n\gs \eta_i(l)$. We note that if one is able to show that the for certain $i$ the image of the function $\eta_i$ is  bounded, then the Noetherian assumption on $ \HH{i}{\fm}{R/I^n}$ for $n\gg 0$ would imply  $\Lambda_i$ is finite, and then by \cite[3.4]{BCH}  the sequence $ \{\indeg \HH{i}{\fm}{R/I^n}\}_{n\in \NN}$ would agree with a linear function for $n\gg 0$.
\end{remark}

\section{Monomial ideals}\label{monoSec}
The purpose of this section is to analyze the asymptotic behavior of $\{\indeg \HH{i}{\fm}{R/I^n}\}_{n\in \NN}$ for monomial ideals.
From now on we assume $R=k[x_1,\ldots, x_d]$, $\fm=(x_1,\ldots, x_d)$, and $I$ is a monomial ideal.

%In this section we show that Question \ref{motivQ} has a positive answer for wide family of monomial ideals. 

\begin{Notation}\label{takNot}
Let $F$ be a subset of $[d]= \{1,\ldots, d\}$. We consider the map $\pi_F: R\longrightarrow R,$ defined by $\pi_F(x_i)=1$ if $i\in F$, and $\pi_F(x_i)=x_i$ otherwise.  We set $I_F:=\pi_F(I)$. 
 For $\fa=(a_1,\ldots, a_d)\in \NN^d$, we use the notation $\fx^\fa := x_1^{a_1}\cdots x_d^{a_d}$.    We also consider $G_\fa=\{i\mid a_i<0\}$ and define  $\fa^+=(a_1^+,\ldots, a_d^+)$, where $a_i^+=a_i$ if $i\not\in G_\fa $ and $a_i^+=0$ otherwise.  We set $\Delta_{\fa}(I)$ to be the simplicial complex of all subsets $F$ of $[d]\setminus G_\fa$ such that $\fx^{\fa^+}\not\in I_{F\cup G_\fa}$.   We note that $\Delta_{\fa}(I)$ is a subcomplex of $\Delta(I)$, the simplicial complex whose Stanley-Reisner ideal  is  $\sqrt{I}$ (\cite[1.3]{MT}). 
\end{Notation}

% \begin{remark}\label{char_of_subcomplex}
%Let $n\in\NN$ and $\fa\in \NN^d$. Given $\Delta'$ a subcomplex of $\Delta(I)$, we have that $\Delta_{\fa}(I^n) = \Delta'$, if and only if $\fx^{\fa}\not \in I^n_F$, for every facet $F$ of $\Delta'$, and $\fx^\fa\in I_G^n$, for every minimal non-face $G$ of $\Delta'$.
%\end{remark}

We now state Takayama's formula which expresses the graded components of local cohomology of monomial ideals in terms of reduced homology of some associated simplicial complexes. 
 
 \begin{thm}[ {\cite[Theorem 1]{Tak}}]\label{Takayama}
 For every $\fa\in \ZZ^d$ and $i\gs 0$ we have
 $$\dim_k \HH{i}{\fm}{R/I}_{\fa} =  \dim_k  \tilde{H}_{i-|G_\fa |-1}(\Delta_{\fa}(I), k)$$
 \end{thm}
 
 The following is the main theorem of this section. 
 
 \begin{thm}\label{monomials}
 Let $I$ be a monomial ideal and assume $\HH{i}{\fm}{R/I^n}$ is Noetherian for $n\gg 0$. Then one of the following holds
 \begin{enumerate}
 \item If $\tilde{H}_{i-1}(\Delta(I),k)\neq 0$, then $\indeg \HH{i}{\fm}{R/I^n} =0$ for $n\gg 0$.
 \item If $\tilde{H}_{i-1}(\Delta(I),k) = 0$ then $\liminf_{n\rightarrow \infty}\frac{\indeg \HH{i}{\fm}{R/I^n} }{n}\gs 1$.
 \end{enumerate}
 \end{thm}
 \begin{proof}
By \cite[Proposition 1]{Tak} and the assumption we have $\HH{i}{\fm}{R/I^n}_{\fa}=0$ for $n\gg 0$, and every $\fa\in \ZZ^d$ such that $G_\fa\neq \emptyset$. In particular,  $\indeg \HH{i}{\fm}{R/I^n} \gs 0$ for every  $n\gg 0$. Set $\f0=(0,\ldots,0)\in \NN^d$. We note that $\Delta_{\f0}(I^n) = \Delta(I)$ for every $n$ (\cite[1.4]{MT}), therefore if $\tilde{H}_{i-1}(\Delta(I),k)\neq 0$, Theorem \ref{Takayama} implies $\indeg \HH{i}{\fm}{R/I^n} =0$ for every $n\gg 0$.

Now, assume $\tilde{H}_{i-1}(\Delta(I),k) = 0$. Fix $n\in \NN$ and $\fa\in \NN^d$ such that $|\fa |<n$. 
For every facet $F$ of $\Delta(I)$ we have $I_F^n\neq 1$, hence by degree reasons $\fx^\fa\not\in I_F^n$. 
It follows $\Delta_\fa(I^n)=\Delta(I)$ and then $\HH{i}{\fm}{R/I^n}_{\fa}=0$ by Theorem \ref{Takayama}. We conclude $\indeg \HH{i}{\fm}{R/I^n}\gs n$, finishing the proof. 
 \end{proof}

\begin{remark}
The condition $\tilde{H}_{i-1}(\Delta(I),k)\neq 0$ in Theorem \ref{monomials} (1) is automatically satisfied if $\HH{i}{\fm}{R/I^n}$ is Noetherian for some $n\in \NN $ and $\HH{i}{\fm}{R/\sqrt{I}}\neq 0$ (\cite[4.9]{DDM}). 
\end{remark}

The following example answers Question \ref{motivQ}, (2) in the particular case that $\Delta(I)$ is a cycle graph $\C_d$ for $d\gs 5$.

\begin{example}
Let $d\gs 5$ and $\C_d$ be the cycle graph of length $d$, i.e., the edges of $\C_d$ are indexed by $\{i,i+1\}$ for $1\ls i\ls d$ where $\{d,d+1\}:=\{d,1\}$. Let $I$ be the Stanley-Reisner ideal of $R$ associated to the complex $\Delta(I)=\C_d$. Then 
\begin{equation}\label{limitExMon}
\lim_{n\rightarrow \infty}\frac{\indeg\HH{1}{\fm}{R/I^n}}{n}=1.
\end{equation}
To show this, notice that since $\Delta(I)$ is connected, we have $\tilde{H}_0(\Delta(I))=0$. Then by Theorem \ref{monomials} it suffices to show $\displaystyle\limsup_{n\rightarrow \infty}\frac{\indeg\HH{1}{\fm}{R/I^n}}{n}\ls 1$. For each $1\ls i\ls d$, set $\fp_i = (\{x_j\mid j\neq i,i+1\})$. Hence, $\Ass(I)=\{\fp_1,\ldots, \fp_d\}$.

Note that 
$$I_{\{1\}}=(x_3,\ldots,x_n)\cap(x_2,\ldots, x_{n-1})=(x_2x_n, x_3,x_4,\ldots, x_{n-1})$$ 
which is a complete intersection. Likewise, $I_{\{i\}}$ is a complete intersection for every $1\ls i\ls d$, therefore $\Proj R/I$ is lci. Hence, $\HH{1}{\fm}{R/I^n}$ is Noetherian for every $n\in \NN$ and  we have $$\tilde{I^n}=(I^n:_R \fm^\infty ) = \cap_{i=1}^d \fp_i^n .$$
Fix $n\gg 0$ and let $\fa_n = (n-d+4, 0, 1, \ldots, 1, 0)\in \NN^d $, then  one readily verifies $\Delta_{\fa_n}(\tilde{I^n})$ is the subcomplex of $\Delta(I)$ whose facets are $\{i,i+1\}$ for $i\neq 2, d-1$. Since $\Delta_{\fa_n}(\tilde{I^n})$ is disconnected, we have $\dim_k \HH{1}{\fm}{R/\tilde{I^n}}_{\fa_n} = \dim_k \tilde{H}_0(\Delta_{\fa_n}, k)\neq 0$. We conclude $$\indeg  \HH{1}{\fm}{R/\tilde{I^n}}\ls |\fa_n|=n+1.$$ Finally, the limit \eqref{limitExMon} follows by noticing $\HH{1}{\fm}{R/I^n} \cong \HH{1}{\fm}{R/\tilde{I^n}}$.
\end{example}

\section*{Acknowledgments}

The authors are grateful to David Eisenbud, Jack Jeffries, Robert Lazarsfeld, Luis N\'u\~nez-Betan\-court, and Claudiu Raicu for very helpful discussions. They also thank Tai H\`a for bringing the reference \cite{BCH} to their attention. Part of the research included in this article was developed in the Mathematisches Forschungsinstitut Oberwolfach (MFO) while the authors were in residence at the institute under the program  {\it Oberwolfach Leibniz Fellows}. The authors thank MFO for their hospitality and excellent conditions for conducting research. The authors would also like to thank the referee for her or his helpful comments and suggestions that improved this paper.


\begin{thebibliography}{99}
%\addcontentsline{toc}{section}{References}

\bibitem{BCH} A. Bagheri, M. Chardin, and H.T. H\`a, \emph{The eventual shape of Betti tables of powers of ideals}, Math. Res. Lett. {\bf 20} (2013), 1033--1046.
%\bibitem{BS} D. Bayer and M. Stillman, \emph{A criterion for detecting m-regularity},  Invent. Math. {\bf 87} (1987), 1--11.



\bibitem{BBLSZ} B. Bhatt, M. Blickle, G. Lyubeznik, A. K. Singh, and W. Zhang, \emph{Stabilization of the cohomology of thickenings}, Amer. J. Math. {\bf 141} (2019), 531--561

%\bibitem{Br} M. Brodmann, \emph{The asymptotic nature of the analytic spread}, Math. Proc. Cambridge Philos. Soc. {\bf 86} (1979), 35--39.

%\bibitem{Br2} M. Brodmann,  \emph{Asymptotic stability of $Ass(M/I^nM)$},  Proc. Amer. Math. Soc. {\bf 74} (1979), 16--18.

%\bibitem{BLS} M. Brodmann, C. H. Linh, and M-H Seiler, \emph{Castelnuovo-Mumford regularity of annihilators, Ext and Tor modules}, In: Commutative algebra, 207--236, Springer, New York, 2013. 

%\bibitem{BrSh} M.P. Brodmann and R.Y. Sharp, \emph{Local cohomology. An algebraic introduction with geometric applications}, Cambridge Studies in Advanced Mathematics, {\bf 136} Cambridge University Press, Cambridge, 2013.



\bibitem{BH} W.~ Bruns and J.~Herzog, \emph{Cohen-Macaulay Rings}, Cambridge Studies in Advanced Mathematics, {\bf 39}, Cambridge, Cambridge University Press, 1993.

%\bibitem{BI} W.~Bruns and B.~Ichim, \emph{On the coefficients of Hilbert quasipolynomials},  Proc. Amer. Math. Soc. {\bf 135} (2007), 1305--1308.

%\bibitem{Normaliz} W.~Bruns, B.~Ichim, T.~Römer, R.~Sieg, and C.~Söger: \emph {Normaliz. Algorithms for rational cones and affine monoids}. Available at https://www.normaliz.uni-osnabrueck.de.

%\bibitem{CutAdv} S. D.~Cutkosky, \emph{Asymptotic multiplicities of graded families of ideals and linear series}, Adv. Math. {\bf 264} (2014), 55--113.

\bibitem{CHH} M. Chardin, D. T. Ha, and L. T. Hoa, \emph{Castelnuovo–Mumford regularity of Ext modules and homological degree}, Trans. Amer. Math. Soc. {\bf 363} (2011), 3439--3456.

\bibitem{CHT} S. D. Cutkosky, J. Herzog, and N.V. Trung, \emph{Asymptotic behaviour of Castelnuovo–Mumford regularity}, Compos. Math. {\bf 80} (1999),  273--297.

%\bibitem{CHST} S. D.~Cutkosky, H. T.~H\`a, H.~Srinivasan, and E.~Theodorescu, \emph{Asymptotic behavior of the length of local cohomology}, Canad. J. Math. {\bf 57} (2005), 1178--1192.

\bibitem{DDM} H.~Dao, A.~De Stefani, and L.~Ma,  \emph{Cohomologically full rings}, preprint (2018) arXiv: 1806.00536.



\bibitem{DM} H.~Dao and J.~Monta\~no, \emph{Length of local cohomology of powers of ideals}, Trans. Amer. Math. Soc. {\bf 371} (2019), 3483--3503.

%\bibitem{DS} H.~Dao and I.~Smirnov, \emph{On generalized Hilbert-Kunz function and multiplicity}, 1305.1833 (2013).

%\bibitem{DRT} T.T.~Dinh, M.E.~Rossi and N.V.~Trung, \emph{Castelnuovo-Mumford regularity and Ratliff-Rush closure}, preprint (2015), {\tt arXiv:1512.04372}.

%\bibitem{DEP} C. De Concini, D. Eisenbud, and C. Procesi, \emph{Young diagrams and determinantal varieties}, Invent. Math. {\bf 56} (1980),  129--165.

\bibitem{E} D.~Eisenbud, \emph{Commutative Algebra: with a view toward algebraic geometry}, Graduate Texts in Mathematics, 150, Springer-Verlag, New York, 1995.

\bibitem{EHU} D. Eisenbud, C. Huneke, and B. Ulrich, \emph{What is the Rees algebra of a module?}, Proc. Amer. Math. Soc. {\bf 131} (2003), 701--708.

\bibitem{EMS} D. Eisenbud, M. Musta\c{t}\u{a}, and M. Stillman, \emph{Cohomology on toric varieties and local cohomology with monomial supports}, J. Symbolic Comput. {\bf 29} (2000), 583--600.

\bibitem{GS} D.~Grayson and M.~Stillman, \emph{Macaulay2, a software system for research in algebraic geometry}, Available at http://www.math.uiuc.edu/Macaulay2.


%\bibitem{HH} J. Herzog and T. Hibi, \emph{Monomial ideals}, Springer London, 2011.


%\bibitem{HHT} J.~Herzog, T.~Hibi and N.V.~Trung, \emph{Symbolic powers of monomial ideals and vertex cover algebras}, Adv. Math. {\bf 210} (2007), 304--322.

\bibitem{HPV} J. Herzog, T. J. Puthenpurakal, and J. K. Verma, \emph{Hilbert polynomials and powers of ideals}, Math. Proc. Cambridge Philos. Soc. {\bf 145} (2008), 623--642.


%\bibitem{HRV} J. Herzog, A. Rauf, and M. Vladoiu, \emph{The stable set of associated prime ideals of a polymatroidal ideal}, J. Algebraic Combin. {\bf 37} (2013), 289–-312.


%\bibitem{HV} J.~Herzog and M.~Vladoiu, \emph{Monomial ideals with primary components given by powers of monomial prime ideals}, Electron. J. Combin., {\bf 21} (2014), 1-18.

%\bibitem{HKTT} L. T. Hoa, K. Kimura, N. Terai and T. N. Trung, \emph{Stability of depths of symbolic powers of Stanley-Reisner ideals}, J. Algebra {\bf 473} (2017), 307--323.

%\bibitem{HS} C.~Huneke and I.~Swanson, \emph{Integral closures of ideals, rings and modules}, London Math. Society Lecture Note Series {\bf 336}, Cambridge University Press, 2006.

%\bibitem{JM} J. Jeffries and J. Monta\~no, \emph{The j-multiplicity of monomial Ideals}, Math. Res. Lett. {\bf 20} (2013), 729–744.

\bibitem{Hor} G. Horrocks., \emph{Vector bundles on the punctured spectrum of a local ring}, Proc. London Math. Soc. (3) {\bf 14} (1964), 689--713. 

%\bibitem{KV} D. Katz and J. Validashti, \emph{Multiplicity and Rees valuations}, Collect. Math. {\bf 61} (2010), 1--24.

\bibitem{Kod} V.~Kodiyalam, \emph{Asymptotic behaviour of Castelnuovo-Mumford regularity}, Proc. Amer. Math. Soc. {\bf 128} (2000), 407--411.

%\bibitem{Kunz} E.~Kunz, \emph{Characterizations of regular local rings for characteristic p}, Amer. J. Math. {\bf 91} (1969), 772–-784. 

%\bibitem{LT}  H.M. Lam and N.V. Trung, \emph{Associated primes of powers of edge ideals and ear decompositions of graphs}, arXiv:1506.01483 (2015).

%\bibitem{Laksov} D. Laksov, \emph{Divided powers}, unpublished, (2006).

%\bibitem{Ly} G. Lyubeznik, \emph{Finiteness properties of local cohomology modules (an application of D- modules to commutative algebra)}, Invent. Math. {\bf 113} (1993), 41--55.

%\bibitem{Linquan1} L. Ma and P. H. Quy, \emph{Frobenius actions on local cohomology modules and deformation"}, preprint arXiv:1606.02059 (2016).

%\bibitem{Mc} P.~McMullen, \emph{Lattice invariant valuations on rational polytopes}, Arch. Math. (Basel), {\bf 31} (1978), 509--516.

\bibitem{MT} N.C. Minha and N.V. Trung, \emph{Cohen–Macaulayness of powers of two-dimensional squarefree monomial ideals}, J. Algebra {\bf 322} (2009), 4219--4227.

\bibitem{Claudiu}  C.~Raicu, \emph{Regularity and cohomology of determinantal thickenings}, Proc. London Math. Soc. (3)  {\bf 116} (2018), 248--280.

\bibitem{Raicu2} C.~Raicu, \emph{Homological invariants of determinantal thickenings}, 
Bull. Math. Soc. Sci. Math. Roumanie, {\bf 60} (108) (2017), 425--446.

\bibitem{Roby} N. Roby, \emph{Lois polynomes et lois formelles en théorie des modules}, Ann. Sci. \'Ecole Norm. Sup.  {\bf 80} (1963), 213--348. 


%\bibitem{Saeden} G. S\ae d\'en St\r{a}hl, \emph{An intrinsic definition of the Rees algebra of a module}, Proc. Edinb. Math. Soc., to appear.

\bibitem{Sba} E.~Sbarra, \emph{Upper bounds for local cohomology for rings with given Hilbert function}, Comm. Algebra {\bf 29} (2001), 5383–-5409.


%\bibitem{Schwede}  K. Schwede, \emph{F-injective singularities are Du Bois}, Amer. J. Math. {\bf 131} (2009), 445–473.


%\bibitem{SVV} A. Simis, W. Vasconcelos, R. Villarreal, \emph{On the ideal theory of graphs}, J. Algebra {\bf 167} (1994), 389--416.

%\bibitem{St} R.P.~Stanley, \emph{Decompositions of rational convex polytopes}, Ann. Discrete Math. {\bf 6} (1980), 333--342.


%\bibitem{SW} A. Singh and U. Walther, \emph{Local cohomology and pure morphisms}, Illinois Journal of Mathematics, Special Issue in Honor of Phil Griffith, {\bf 51} (2007), 287--298.

\bibitem{Tak} Y. Takayama, \emph{Combinatorial characterizations of generalized Cohen-Macaulay monomial ideals}, Bull. Math. Soc. Sci. Math. Roumanie (N.S.), {\bf 48} (2005), 327–-344.

\bibitem{TW} N.V.~Trung and H.J.~Wang,  \emph{On the asymptotic linearity of Castelnuovo–Mumford regularity}, J. Pure Appl. Algebra {\bf 201} (2005), 42--48.

% \bibitem{UV} B. Ulrich and J. Validashti, \emph{Numerical criteria for integral dependence}, Math. Proc. Cam- bridge Philos. Soc., {\bf 151} (2011), 95–-102.

%\bibitem{Vas} W.~Vasconcelos, \emph{Integral closure: Rees algebras, multiplicities, algorithms}, Springer Monographs in Mathematics. Springer, New York, 2005.

%\bibitem{Wald}  M. Waldschmidt, Propri \'et \'es arithm \'etiques de fonctions de plusieurs variables. II. In S\'eminaire P. Lelong (Analyse), 1975/76, Lecture Notes Math. 578, Springer, 1977, 108–135.

%\bibitem{Woods} K. Woods, \emph{Presburger arithmetic, rational generating functions, and quasi-polynomials}, J. Symb. Log. {\bf 80} (2015), 433--449.

%\bibitem{Zhang} W.~Zhang, \emph{A note on the growth of regularity with respect to Frobenius}, preprint, arXiv:1512.00049 (2015).

\end{thebibliography}
\end{document}